\RequirePackage{fix-cm}

\documentclass[a4paper,12pt]{amsart}
\usepackage[a4paper, margin=1in]{geometry}     

\usepackage{amsthm}
\usepackage{aliascnt}

\RequirePackage[dvipsnames,usenames]{xcolor}
\usepackage{mathtools}
\usepackage[all]{xy}
\usepackage{tikz}
\usepackage{tikz-cd}
\usepackage{systeme}
\usepackage{tikz-cd}
\usepackage{amsmath,amssymb,amsfonts}
\usepackage[
backend=bibtex,
style=numeric,
citestyle=numeric
]{biblatex}

\usepackage{hyperref}
\usepackage[shortlabels]{enumitem}

\newcommand{\basetheorem}[3]{
	\newtheorem{#1}{#2}[#3]
	\newtheorem*{#1*}{#2}
	\expandafter\def\csname #1autorefname\endcsname{#2}
}
\newcommand{\maketheorem}[3]{
	\newaliascnt{#1}{#3}
	\newtheorem{#1}[#1]{#2}
	\aliascntresetthe{#1}
	\expandafter\def\csname #1autorefname\endcsname{#2}
	\newtheorem{#1*}{#2}
}

\basetheorem{theorem}{Theorem}{section}
\maketheorem{definition}{Definition}{theorem}
\maketheorem{lemma}{Lemma}{theorem}
\maketheorem{example}{Example}{theorem}
\maketheorem{corollary}{Corollary}{theorem}
\maketheorem{conjecture}{Conjecture}{theorem}
\maketheorem{proposition}{Proposition}{theorem}
\maketheorem{question}{Question}{theorem}
\maketheorem{remark}{Remark}{theorem}

\newcommand{\A}{\mathcal{A}}

\newcommand{\F}{\mathcal{F}}

\newcommand{\SB}{\mathbf{SB}}
\newcommand{\BB}{\mathbf{B}}
\newcommand{\BS}{\mathbf{B}_{+}}

\newcommand{\Ht}[1]{H^{i}_{t}}

\newcommand{\perf}{\textup{perf}}

\newcommand{\ox}[1][X]{\mathcal{O}_{#1}}

\newcommand{\stacks}[1]{\cite[\href{https://stacks.math.columbia.edu/tag/#1}{Tag #1}]{stacks-project}}

\begin{document}
	
	\title{The Augmented Base Locus in Mixed Characteristic}

\author{Liam Stigant}
\address{Department of Mathematics, Imperial College London, 180 Queen's Gate, 
	London SW7 2AZ, UK} 
\email{l.stigant18@imperial.ac.uk}
\date{\today}
\pagestyle{myheadings} \markboth{\hfill  Liam Stigant
	\hfill}{\hfill the Augmented Base Locus in Mixed Characteristices\hfill}
	\maketitle
	
	\begin{abstract}
		Let $L$ be a nef and big line bundle on a scheme $X$. It is well known that if $X$ is a projective over a field then the augmented base locus and the exceptional base locus agree. This result is extended to projective schemes over arbitrary excellent Noetherian bases, assuming the result holds in characteristic zero. In particular the result holds if the base is a mixed characteristic Dedekind domain or if $L$ is semiample in characteristic $0$.
	\end{abstract}

\section{Introduction}

The augmented base locus is well studied for schemes over a field. An important characterisation, first noted for smooth varieties of characteristic $0$ by Nakayame \cite{nakamaye2000stable}, is that for a nef line bundle $L$ the augmented base locus $\BS(L)$ agrees with the exceptional locus $\mathbb{E}(L)$. 

Since then the result has been shown to hold for projective schemes over a field, first in positive characteristic by Cascini-M\textsuperscript{c}Kernan-Musta\c{t}\v{a} \cite{cascini2014augmented}, and then for $\mathbb{R}$-divisors over any field by Birkar \cite{birkar2017augmented}. Similar results are given for non-nef divisors in \cite{ein2009restricted} and for K\"{a}hler manifolds in \cite{collins2015kahler}.

We make use of methods developed in \cite{witaszek2020keel} together with ideas from the positive characteristic proof to show that $\BS(L)=\mathbb{E}(L)$ for a nef line bundle on a projective scheme over an excellent Noetherian base, so long it holds true on the characteristic zero part the scheme. In particular the result holds in the following cases.

\begin{theorem}[\autoref{main1}]
Let $X$ be a projective scheme over an excellent Noetherian base $S$ with $L$ a nef line bundle on $X$. 
Suppose that one of the following holds:
\begin{enumerate}
	\item $S_{\mathbb{Q}}$ has dimension $0$;
	\item $L|_{X_{\mathbb{Q}}}$ is semiample;
\end{enumerate}

Then $\BS(L)=\mathbb{E}(L)$.
\end{theorem}

We also extend the semiampleness result of \cite{witaszek2020keel} to show that there is an equality of stable base loci when the characteristic $0$ part is semiample.
 
 \begin{theorem}[\autoref{main2}]
 	Suppose that $X$ is a projective scheme over an excellent Noetherian base with $L$ a nef line bundle on $X$. Then $\SB(L)=\SB(L|_{\mathbb{E}(L)})$ so long as $L|_{X_{\mathbb{Q}}}$ is semiample.
 \end{theorem}

\subsection*{Acknowledgements}
	This work was born from fruitful conversations with Fabio Bernasconi, Federico Bongiorno, Iacopo Brivio and Paolo Cascini. Thanks in particular to Fabio and Paolo for their comments and advice on earlier drafts. I am also grateful to the EPSRC for my funding. Thanks also to the referee for their careful reading of an earlier draft and pointing out several issues with it.\\
	\\
	Data sharing not applicable to this article as no datasets were generated or analysed.\\
	\\
	On behalf of all authors, the corresponding author states that there is no conflict of interest.
\section{Preliminaries}

We will work exclusively with line bundles. Since the schemes we work with need not be normal, line bundles are not the same as Cartier divisors, however we typically use the traditional notation for divisors as we still sometimes treat line bundles as Cartier divisors when appropriate. That is we write the tensor product of $L,L'$ as $L+L'$, $L^{\otimes k}$ is often written $kL$ and given $f:Y \to X$, then $f^{*}L=L|_{Y}$ is often written $\ox[Y](L)$, including for $Y=X, f=id$. 

Since the questions considered are local on the base, it suffices to work only with affine bases. In particular, for notational simplicity, $H^{i}(X,L)$ will often be used to denote the higher derived pushforwards of $L$ by $X \to S$. 
	
\begin{definition}
	Let $L$ be a line bundle on a projective Noetherian scheme $X$ over some Noetherian scheme $S$. Then base locus is given as 
	$$\BB(L)= \bigcap_{s \in H^{0}(X,L)} Z(s)_{red}$$
	where $Z(s)$ is the zero set of $s$ equipped with the obvious scheme structure. The stable base locus is then
	$$\SB(L)=\bigcap_{m \geq 0}\BB(mL).$$
	Fix an ample line bundle $A$. The augmented base locus is given as 
	$$\BS(L)=\bigcap_{m \geq 0}\SB(mL-A)$$
	and is independent of the choice of $A$.
	
\end{definition}

We could also write \[\BS(L)= \bigcap_{A \text{ ample, } m \geq 0}\SB(mL-A)\] for a definition that involves no choice of ample line bundle. By Noetherianity if we choose $m$ sufficiently large and divisible then in fact $\BS(L)=\SB(mL-A)$.

 \begin{definition}
	Let $L$ be a line bundle on a projective scheme $X$. The exceptional locus, $\mathbb{E}(L)$, is the union of integral subschemes on which $L$ is not big.
\end{definition}

The previous two definitions are invariant under scaling by $n \in \mathbb{N}_{\geq 0}$ and line bundles will frequently be replaced with higher multiples.

\begin{theorem}\cite{witaszek2020keel}[Theorem 1.10]\label{red}
	Suppose that $X$ is a projective scheme over an excellent Noetherian base $S$ and $L$ is a nef line bundle on $X$. Then if $L|_{X_{red}}$ and $L|_{X_{\mathbb{Q}}}$ are semiample so too is $L$.
\end{theorem}

\begin{theorem}\cite{keeler2003ample}[Theorem 1.5]\label{{Keeler}}
	Let $X$ be a projective scheme over a Noetherian ring, $\A$ an ample line bundle and $\F$ a coherent sheaf. Then there is some $m_{0}$ with 
	
	\[H^{i}(X, \F \otimes \A^{m} \otimes \mathcal{N})=0\]
	for all $i>0, m \geq m_{0}$ and all nef line bundles $\mathcal{N}$.
\end{theorem}

\begin{lemma}\cite{cascini2014augmented}[Lemma 2.2]\label{sec-growth}
	Let $X$ be an n-dimensional projective scheme over a field $k$ and $L$ a line bundle on $X$.
	For every coherent sheaf $\F$ on $X$, there is $C > 0$ such that $h^{0}(X, \F\otimes L^{m}) \leq C m^{n}$ for every $m \geq 1$.
\end{lemma}

\begin{lemma}\label{bigsecs}
	Let $X$ be a reduced projective scheme over a ring $R$ and $L,A$ line bundles on $X$ with $A$ ample. Then for large $m$ and general $s \in H^{0}(X,mL-A)$ and any irreducible component $Y$ of $X$ with $L|_{Y}$ big we have $Y \not \subseteq Z(s)$.
\end{lemma}

\begin{proof}
	
	Let $f:X \to S$ be the structure morphism. Suppose for contradiction that $f_{*}\ox(mL-A) \to f_{*}\ox[Y](mL-A)$ is the zero map for infinitely many $m$.
	
	Let $W$ be the union of the other components of $X$ so that we have a short exact sequence 
	
	\[0 \to \ox \to \ox[Y]\oplus \ox[W] \to \ox[Y\cap W] \to 0\]
	
	where $Y,W$ are given the reduced subscheme structure. For convenience we write $Z=Y\cap W$
	
	Tensoring and pushing forwards we get 
	\[0 \to f_{*}\ox(mL-A)\to f_{*}\ox[Y](mL-A)\oplus f_{*}\ox[W](mL-A) \to f_{*}\ox[Z](mL-A) \]
	
	In particular if $f_{*}\ox(mL-A) \to f_{*}\ox[Y](mL-A)$ is the zero map, we must have an injection $ f_{*}\ox[Y](mL-A) \hookrightarrow f_{*}\ox[Z](mL-A)$. Let $V=f(Y)$ and $g=f|_{Y}:Y \to V$. Then we may view $\ox[Y](mL-A), \ox[Z](mL-A)$ as sheaves on $Y$, then there is a corresponding injection $g_{*}\ox[Y](mL-A) \hookrightarrow g_{*}\ox[Z](mL-A)$ since the pushforward is left exact. Since $Y$ is irreducible so too is $V$ and hence we may pull back to the generic point $\nu$ of $V$.
	
	Now we have that $Y_{\nu}$ is a projective scheme over $K(V)$ of dimension say $n$. Equally $Z_{\nu}$ is a closed subscheme of $Y_{\nu}$ of dimension at most $n-1$. We now find a contradiction by counting sections over $K(V)$.
	
	On the one hand we have an injection $$H^{0}(Y_{\nu},\ox[Y_{\nu}](mL-A)) \hookrightarrow H^{0}(Z_{\nu}, \ox[Z_{\nu}](mL-A)),$$ which ensures that there is $C > 0$ such that $h^{0}(Y_{\nu}, \ox[Y_{\nu}](mL-A)) \leq C m^{n-1}$ for every $m \geq 1$ by \autoref{sec-growth}. On the other, $kL|_{Y_{\nu}}$ is big, and $Y_{\nu}$ is integral, thus $h^{0}(Y_{\nu}, \ox[Y_{\nu}](mL-A))$ grows like $m^{n}$ by \cite[Lemma 4.2]{birkar2017augmented}. This is a contradiction and the result follows.		
\end{proof}

\begin{remark}\label{powers}
	
	When $X$ is a reduced scheme and $X=X_{1} \cup X_{2}$ (as topological spaces) for closed subschemes $X_{1},X_{2}$ we have a short exact sequence 
	\[0 \to \ox \to \ox[X_{1}]\oplus \ox[X_{2}] \to \ox[X_{1}\cap X_{2}] \to 0\]
	as used above. In particular if $L$ is a line bundle on $X$ with sections $s_{1},s_{2}$ on $X_{1},X_{2}$ respectively which agree on $X_{1}\cap X_{2}$ then they glue to a section of $L$ on $X$.
	
	This is not the case when $X$ is reducible. If $X_{j}$ are given by ideal schemes $I_{j}$ then it need not be the case that $I_{1} \cap I_{2} = 0$. However replacing $I_{1}$ with a higher power we may suppose that this is the case, see for instance \stacks{01YC}. In particular we may always choose subscheme structures such that the short exact sequence
	\[0 \to \ox \to \ox[X_{1}]\oplus \ox[X_{2}] \to \ox[X_{1}\cap X_{2}] \to 0\]
	still holds. When we work with components of a reducible scheme we can always chose the subscheme structure in this fashion, and in particular we will always be able to glue appropriate sections.
	
\end{remark}

\begin{lemma}\label{Blowup-close}\cite{eisenbud2000geometry}[Proposition IV-21]
	Let $X$ be a scheme and $Z \subseteq X$ a subscheme with $\pi\colon Y \to X$ the blowup of $X$ along $Z$. If $f\colon X' \to X$ is any morphism and we write $Z'=f^{-1}Z$, then the closure $W$ of $\pi_{X'}^{-1}(X'\setminus Z')$ inside $X' \times_{X} Y$ is exactly the blowup of $X'$ along $Z'$.
\end{lemma}

\begin{lemma}\label{Blowup-red}\stacks{0808}
	Let $X$ be a scheme. Let $I\subseteq \ox$ be a quasi-coherent sheaf of ideals. If $X$ is reduced, then the blowup $X'$ of $X$ along $I$ is reduced.
	
\end{lemma}

Together these tell us that 'the blowup of the reduction is the reduction of the blowup'. More precisely we have the following.

\begin{lemma}\label{Blowup}
	
	Let $X$ be a scheme and $Z$ a proper closed subscheme of $X_{red}$. Let $\pi:X' \to X$ be the blowup of $X$ along $Z$, viewed as a subscheme of $X$. Let $Y$ be the blowup of $X_{red}$ along $Z$, then we have isomorphisms
	
	\[Y \simeq X'\times_{X} X_{red} \simeq X'_{red}\]
	
\end{lemma}
\begin{proof}
	
	First we observe that $X'\times_{X} X_{red} \simeq X'_{red}$. Indeed if $f:Z \to X'$ is a morphism from a reduced scheme, then we have a composition $g=\pi \circ f: Z \to X$. And thus a unique induced morphism $Z \to X_{red}$. By definition this induces a unique morphism $Z \to X'\times_{X} X_{red}$ and hence $X'\times_{X} X_{red}$ satisfies the universal property of the reduced subscheme, ensuring that $X'\times_{X} X_{red} \simeq X'_{red}$.
	
	Now by \autoref{Blowup-close} we have that $Y$ is the closure of $(X_{red}\setminus Z)$ inside $X'\times_{X} X_{red}$. However $X_{red}\setminus Z$ is a dense subscheme and so $Y$ is precisely the reduced subscheme of $X'\times_{X} X_{red}$, but then in fact they are equal as $X'\times_{X} X_{red}$ is already reduced.

\end{proof}

More abstractly one can argue that \autoref{Blowup} holds by appealing directly to the universal properties of the reduction, \stacks{0356} and the blowup, \stacks{0806}.

\begin{lemma}[Elimination of Indeterminacy by blowups]\label{elim}
	Let $f: X \dashrightarrow Y$ be a rational map of $S$ schemes associated to an $S$-linear system $|V|\subseteq H^{0}(X,L)$ without fixed part, then there is $Z$ with maps $\phi_{1}:Z \to X$, $\phi_{2}:Z \to Y$ such that $\phi_{1}^{*}L=M+F$ for $M$ a line bundle globally generated by $\phi_{1}^{*}|V|$. Here $F \geq 0$ is such that $\ox[Y](-F)$ is a line bundle, $\phi_{1}(F)=\BB|V|$ as reduced schemes and $\phi_{2}=f\circ \phi_{1}$. Further we may construct $Z \to X$ as a blowup of $X$.
\end{lemma}

\begin{proof}
	
	Consider the following morphism of line bundles
	$V \otimes L^{-1} \to \ox$
	and let $\mathcal{I}$ be the image. Then $\mathcal{I} \otimes L$ is the image of $V \otimes \ox \to L$, in particular the support of $\mathcal{I}$ is exactly $\BB|V|$.
	
	Let $\pi:Z \to X$ be the blowup of $X$ along $\mathcal{I}$. We then have $\pi^{-1}\mathcal{I}\cdot \ox[Z]=\ox[Z](-F)$ for some $F$ an effective Cartier divisor. Hence we have
	\[\pi^{*}(V\otimes L) \twoheadrightarrow \ox[Z](-F) \hookrightarrow \ox[Z] \]
	where the first map is surjective by right exactness of the pullback functor. Tensoring by $\pi^{*}L$ then gives the following.
	
	\[\pi^{*}(V\otimes \ox[z]) \twoheadrightarrow \pi^{*}L(-F) \hookrightarrow \pi^{*}L \]
	
	In particular the line bundle in the middle, which we may write $M$ is globally generated by sections indexed by $\pi^{*}|V|$ and we have $M=\pi^{*}L(-F)$ by construction. Clearly $\pi(F)$ is the support of $\mathcal{I}$, which is nothing but $\BB|V|$. Since $M$ is globally generated it defines a morphism $\phi_{2}:=\phi_{\pi^{*}|V|}:Z \to Y$ and as $\phi_{1}:=\pi$ is an isomorphism away from $F$ the sections in $\pi^{*}|V|$ agree with those of $|V|$ on this locus. Hence $\phi_{2}$ agrees with $f$ here, that is $\phi_{2}=f\circ \phi_{1}$ as required. 
\end{proof}

\begin{lemma}\label{ampall}
	
	Let $X$ be an $S$ scheme and let $H$ be a very ample divisor on $X$. Suppose that $|V|\subseteq H^{0}(X,H)$ is an $S$-linear system which induces a closed immersion $X \to \mathbb{P}^{V}_{S}$.
	
	Then for $k$ sufficiently large we have that $V^{\otimes k}=H^{0}(X,kH)$.
	
\end{lemma}

\begin{proof}
	Thought of as a subscheme of $\mathbb{P}^{V}_{S}$, $X$ is cut out by an ideal sheaf $\mathcal{I}$. Hence we have 
	\[0 \to \mathcal{I}\otimes \ox[\mathbb{P}_{S}^{V}](k) \to \ox[\mathbb{P}_{S}^{V}](k) \to \ox(kH) \to 0.\]
	Since $H^{1}(\mathbb{P}^{V}, \mathcal{I}\otimes \ox[\mathbb{P}^{V}](k))=0$ for large $k$, we get a surjection $$H^{0}(\mathbb{P}^{V},\ox[\mathbb{P}^{V}](k)) \to H^{0}(X,kH).$$ However, the image of this map is precisely $V^{\otimes k}$ since we have $H^{0}(\mathbb{P}^{V},\ox[\mathbb{P}^{V}](k))=\bigotimes_{1}^{k}H^{0}(\mathbb{P}^{V},\ox[\mathbb{P}^{V}](1))$.
	
\end{proof}

\begin{remark}
	
	They key point of this lemma is the following. Suppose we take $|V|$ as in \autoref{elim} on $X$. Then we have a blowup $\phi_{1}^{*} \colon Z \to X$ such that $\phi_{1}^{*}|V|$ is basepoint free inside $H^{0}(Z,M)$. Take the induced morphism $\phi_{2}\colon Z \to Y$ and let $H$ be the very ample divisor on $Y$ induced by $|V|$. Then we have $\phi_{2}^{*}H^{0}(Y,kH) \subseteq \phi_{1}^{*}|V|^{\otimes k}$ for $k>>1$.
	
	This may not be true for $k=1$, even without the resolution of indeterminacy. Consider for example $X=\mathbb{P}^{1}$ and $L=\ox(4)$. If we take $$|V|=<x^4,x^3y,y^3x,y^{4}>$$ then we get an induced morphism $X \to \mathbb{P}^{3}$. The image, $Y$, is not projectively normal however, since $X \to Y$ is an isomorphism but $\dim |V| =4$ and $\dim H^{0}(X,L)=5$. In this example $k=3$ suffices.
	
\end{remark}

\section{Stable Base Loci}

In this section we will examine the stable base locus of line bundles which are semiample over $\mathbb{Q}$. This is then applied to the case of a big and nef line bundle restricted to its exceptional locus.  We begin with an extension of \cite[Theorem 1.10]{witaszek2020keel}. The proof follows the same structure, however more care is needed to keep track of sections.

If $L$ is a line bundle on $X$, semiample over $\mathbb{Q}$, we would like to claim that $\SB(L)=\SB(L|_{X_{red}})$. If $L|_{X_{red}}$ is semiample then this follows from \cite[Thereom 1.10]{witaszek2020keel}. We would then like to prove the general case by blowing up the base locus of $L|_{X_{red}}$ and reducing to the case that the line bundle is semiample on the reduction. Unfortunately if $Y \to X$ is a blowup then the pullback map $H^{0}(X,L) \to H^{0}(Y,\pi^{*}L)$ is, in general, neither injective nor surjective if $X$ is not integral. It is the lack of surjectivity that causes the issues, since we ultimately wish to show the existence of sections on the original scheme.

Suppose for example $X$ is the union of two normal projective schemes $X_{1}$, $X_{2}$. Then if $\pi:Y \to X$ is the blowup of $X_{2}$, the map factors through the closed immersion $X_{1} \hookrightarrow X$. Of course if $L$ is a line bundle on $X$ then $H^{0}(X,L) \to H^{0}(X_{1},L|_{X_{1}})\simeq H^{0}(Y,\pi^{*}L)$ is typically not a surjection.

 The idea in \cite[Thereom 1.10]{witaszek2020keel} is essentially to show that $L$ is semiample by producing a candidate morphism via pushout. Then one can lift sections back to $L$ by building them from suitable sections of $L|_{X_{red}}$ and $L|_{X_{\mathbb{Q}}}$, up to perhaps replacing the line bundle with a higher power. The key remedy, then, is to show that if we blow up the base locus of $L|_{X_{red}}$ via $\pi:Y \to X$, we may build sections of $\pi^{*}L$ on $Y$ using only those coming from $X_{red}$ and $X_{\mathbb{Q}}$.\\

\begin{theorem}\label{BaseRed}
		Let $S$ be an excellent, Noetherian scheme, take $X$ a projective scheme over $S$ and $L$ a line bundle on $X$. Write $i:X_{red} \to X$ for the inclusion of the reduced scheme. Suppose that $L|_{X_{\mathbb{Q}}}$ is semiample. Then $\SB(L)=\SB(L|_{X_{red}})$.
\end{theorem}

\begin{proof}
	We always have $\SB(L|_{X_{red}}) \subseteq \SB(L)$ since we can pull back sections of $L$, so it suffices to show the converse. We may also freely localise on $S$ and assume that it is an affine, Noetherian $\mathbb{Z}_{(p)}$ scheme. After replacing $L$ with a sufficiently high mulitple, we assume that $\SB(L)=\BB(L)$, $\SB(L|_{X_{red}})=\BB(L|_{X_{red}})$ and $\SB(L|_{X_{\mathbb{Q}}})=\BB(L|_{X_{\mathbb{Q}}})$ as reduced schemes.\\
	\\
	\textbf{Step 1: Blow-up the base locus.}\\
	
	Fix a generating set $s_{i}$ of $H^{0}(X_{red},L|_{X_{red}})$. By \autoref{elim} the blowup $W\to X_{red}$ along a subscheme $Z$ eliminates the indeterminacy of $L_{red}$, where $Z=\BB(L|_{X_{red}})=\SB(L|_{X_{red}})$ as reduced schemes. Let $\pi:Y \to X$ be the blowup along $Z$, viewed here a subscheme of $X$. Then the reduction of $Y$ is $Y_{red}\simeq W$ by \autoref{Blowup}.
	
	Let $F$ be the exceptional divisor and $M = \pi^{*}L(-F)$. Note that since $L$ is semiample on $X_{\mathbb{Q}}$, we have that $Y_{\mathbb{Q}}=X_{\mathbb{Q}}$ and $M|_{Y_{\mathbb{Q}}}=L|_{X_{\mathbb{Q}}}$ under this identification. We fix a generating set $t_{i}$ of of $H^{0}(Y_{\mathbb{Q}},M|_{Y_{\mathbb{Q}}})$, which induces a morphism $\phi_{\mathbb{Q}}\colon Y_{\mathbb{Q}} \to Z'_{\mathbb{Q}}$.
		
	By definition the basis $s_{i}$ of $H^{0}(X_{red},L|_{X_{red}})$ now induces $\hat{s}_{i}$ in $H^{0}(Y_{red},M|_{Y_{red}})$ which globally generate the line bundle. These sections induce a morphism $\psi: Y_{red} \to Z$ over $S$. Note that this may not be the same as the morphism induced by the full basepoint free linear system $H^{0}(Y_{red},M|_{Y_{red}})$ since we need not have $H^{0}(Y_{red},M|_{Y_{red}})\simeq H^{0}(X_{red},L|_{X_{red}})$ when $X$ is not irreducible.
	
	We then have an induced morphism $Z_{\mathbb{Q}} \to Z'_{\mathbb{Q}}$ which is a finite universal homeomorphism by \stacks{02OG}. We write $S=\pi^{*}_{red}H^{0}(X,L|_{X_{red}}) \subseteq H^{0}(Y_{red},M|_{Y_{red}})$, which is generated by the $\hat{s}_{i}$ by construction. 
	
	Now by \cite[Theorem 1.7, Corollary 4.20 and Lemma 2.20]{witaszek2020keel}, there is a scheme $Z'$, a universal homeomorphism $Z \to Z'$ and a line bundle $H'$ on $Z$ such that the following diagram commutes at the level of line bundles.
	
	\[\begin{tikzcd}
	{(Y,M)}                                 & {(Y_{\mathbb{Q}},M|_{Y_{\mathbb{Q}}})} \arrow[ddd, bend left=70, "\phi_{\mathbb{Q}}"] \arrow[l]     \\
	{(Y_{red},M|_{Y_{red}})} \arrow[d, "\psi"] \arrow[u] & {(Y_{red, \mathbb{Q}},M|_{Y_{red, \mathbb{Q}}})} \arrow[d, "\psi_{\mathbb{Q}}"] \arrow[u] \arrow[l] \\
	{(Z,H)} \arrow[d]                       & {(Z_{\mathbb{Q}},H|_{Z_{\mathbb{Q}}})} \arrow[l] \arrow[d]                    \\
	{(Z',H')}                               & {(Z'_{\mathbb{Q}},H'|_{Z'_{\mathbb{Q}}})} \arrow[l]                            
	\end{tikzcd}\]\\
	\\
	\textbf{Step 2: Find compatible sections.}\\
	
	Since $\psi$ is not induced by the full linear system on $Y_{red}$, it need not be the case that sections of $H^{0}(Z,H)$ pull back to sections inside the linear system $S\subset H^{0}(Y_{red},M|_{Y_{red}})$ which defines $\psi$. By \autoref{ampall} however, we may replace $M,L,S,H,H'$ with higher multiples so that $\psi^{*}H^{0}(Z,H)\subseteq S$. and $\phi_{\mathbb{Q}}^{*}H^{0}(Z_{\mathbb{Q}}',H'|_{Z_{\mathbb{Q}}'})\subseteq H^{0}(Y_{\mathbb{Q}},M|_{Y_{\mathbb{Q}}})$. 
	Taking further powers as needed, we may suppose also that $H'$ is very ample.
	
	We fix $u_{i}$ a generating set for $H^{0}(Z',H')$, then let $v_{i}=u_{i}|_{Z}$ and $w_{i}=u_{i}|_{Z_\mathbb{Q}}$. By construction we have $\pi^{*}v_{i}\subseteq S$ so we can choose $x_{i}\in H^{0}(X_{red},L|_{X_{red}})$ with $\pi^{*}x_{i}=\psi^{*}v_{i}$. Similarly we have $y_{i} \in H^{0}(X_{\mathbb{Q}},L|_{X_{\mathbb{Q}}})=H^{0}(X_{\mathbb{Q}},M|_{X_{\mathbb{Q}}})$ with $\phi_{\mathbb{Q}}^{*}t_{i}=y_{i}$. Since the above diagram commutes we have the following identifications. $$\pi^{*}v_{i}|_{Y_{red,\mathbb{Q}}}=\psi_{\mathbb{Q}}^{*}(u_{i}|_{Z_{red,\mathbb{Q}}})=y_{i}|_{Y_{red,\mathbb{Q}}}$$
	
	Since $H'$ is very ample the $y_{i}$ must generate a basepoint free linear system. Similarly the $\phi^{*}v_{i}$ are basepoint free on $Y_{red}$. Then as $\pi^{-1}\colon X \dashrightarrow Y$ is an isomorphism away from $\SB(L|_{X_{red}})$, the $x_{i}$ are basepoint free away from it also.
	
	Finally note that since $H^{0}(X_{red,\mathbb{Q}},L|_{X_{red,\mathbb{Q}}}) \to H^{0}(Y_{red,\mathbb{Q}},M|_{Y_{red,\mathbb{Q}}})$ is an isomorphism, we must have $x_{i}|_{X_{red,\mathbb{Q}}}=y_{i}|_{X_{red,\mathbb{Q}}}$.\\
	\\
	\textbf{Step 3: Glue sections on the original scheme.}\\

	By \cite[Proposition 3.5]{witaszek2020keel}, we have the following commutative diagram. 
	
	\[
	\begin{tikzcd}
	{H^{0}(X,L)^{\perf}} \arrow[d] \arrow[r]                 & {H^{0}(X_{\mathbb{Q}},L|_{X_{\mathbb{Q}}})^{\perf}} \arrow[d] \\
	{H^{0}(X_{red},L|_{X_{red}})^{\perf}} \arrow[r] \arrow[r] & {H^{0}(X_{red,\mathbb{Q}},L|_{X_{red,\mathbb{Q}}})^{\perf}}  
	\end{tikzcd}	
	\]
	Hence we can again replace $L$ with a higher power, and $x_{i}$, $y_{i}$ with the corresponding multiples, such that there are $r_{i} \in H^{0}(X,L)$ with $r_{i}|_{X_{red}}=x_{i}$ and $r_{i}|_{X_{\mathbb{Q}}}=y_{i}$. Once again then $L$ is globally generated by the $r_{i}$ away from $\SB(L|_{X_{red}})$, so we must have that $\SB(L)\subseteq \SB(L|_{X_{red}})$ as claimed.
\end{proof}

\begin{remark}
	
	In principle the condition that $L|_{X_{\mathbb{Q}}}$ is semiample is not completely necessary. The blowup of $\BB(L|_{X_{red}})$, $\pi: Y \to X$ induces an injection $$H^{0}(X|_{red,\mathbb{Q}},L|_{X_{red,\mathbb{Q}}}) \to H^{0}(Y|_{red,\mathbb{Q}},L|_{Y_{red,\mathbb{Q}}})$$ which is sufficient to allow us to glue sections on the base. Much more care must be taken when replacing $L$ with a higher power in this case, however, to ensure that the pullback remains injective on the considered sections. 
		
	This would extend the result to the case that $L|_{X_{\mathbb{Q}}}$ becomes basepoint free after we blowup the base locus of $L|_{X_{red}}$. However, it is not clear how this condition could be verified in practice.
	
\end{remark}

We now consider the stable base locus of a big and nef line bundle on restriction to its exceptional locus, under the assumption that the characteristic $0$ part of the line bundle is semiample.

\begin{lemma}
	Let $L$ be a nef line bundle on $X$ projective over an excellent Noetherian base $S$ with and $D$ an effective Cartier divisor such that $L(-D)$ is an ample line bundle. If $L|_{D_{\mathbb{Q}}}$ is semiample then \[\SB(L)=\SB(L|_{D}).\]
\end{lemma}	

\begin{proof}
	Clearly $\SB(L) \subseteq D$ as $L$ is ample away from $D$ and we have $\SB(L|_{D}) \subseteq \SB(L)$ by restriction. Consider the following short exact sequence.
	
	\[0 \to \ox (kL-mD) \to \ox(kL) \to \ox[mD](kL) \to 0\]
	By \autoref{{Keeler}}, we may choose $m >>0$ such that $$H^{1}(\ox,kL-mD=mA+(k-m)L)=0$$ for $k \geq m$. Then by \autoref{BaseRed} and the semiampleness assumption, we have $\SB(L|_{D})=\SB(L|_{mD})$ and may pick $k>> m$ with $\SB(L|_{D})=B(kL|_{mD})$ as reduced subschemes of $X$. In particular if $P$ is any closed point of $D$, we may find a section of $kL|_{mD}$ avoiding it, and then lift this to a section of $kL$. Thus $\SB(L)\cap D \subseteq \SB(L|_{D}) $ and the result follows.

\end{proof}

\begin{lemma}
	Suppose that $X$ is a reduced projective scheme over an excellent Noetherian base. Suppose that $L,A$ are line bundles with $L$ nef and $A$ ample. Take $Z=Z(s)$ for some section $s$ of $L-A$. If $L|_{D_{\mathbb{Q}}}$ is semiample then $\SB(L)=\SB(L|_{Z})$.
\end{lemma}

\begin{proof}
	Let $Y_{1}$ be the union of components of $X$ contained in $Z$ and $Y_{2}$ the union of those not contained in $Z$. If either are empty the result is clear so suppose otherwise. As in \autoref{powers}, we give them a subscheme structure and replace $L,A,s$ with higher powers to ensure we may glue appropriate sections.
	
	Let $D=Z \cap Y_{2}$ and $L_{2}=L|_{Y_{2}}$. By assumption $D$ is a Cartier divisor on $Y_{2}$ with $D=(L-A)|_{Y_{2}}$. As above we have 
	\[0 \to \ox[Y_{2}] (kL_{2}-mD) \to \ox[Y_{2}](kL_{2}) \to \ox[mD](kL_{2}) \to 0\]
	and choosing $k > m >>0$ this allows us to lift sections from $kL_{2}|_{mD}$. We then have $\BB(kL|_{mZ})=\SB(L|_{mZ})=\SB(L|_{Z})= \BB(kL|_{Z})$ for large enough $k$ by \autoref{BaseRed}. Now, given any section $t$ of $kL|_{mZ}$ we may restrict it to $D$ and then lift it to $t'$ a section of $kL_{2}$. By construction $t'$ agrees with $t$ on $D=Z \cap Y_{2}$, and since $Y_{1} \subseteq Z$ it follows we may glue $t|_{Y_{1}}$ and $t'$. In particular then we must have $\SB(L)\cap Z = \SB(L|_{Z})$, but since $L$ is ample away from $Z$ the result follows.
\end{proof}

\begin{corollary}\label{main2}
	Suppose that $X$ is a projective scheme over an excellent Noetherian base with $L$ a nef line bundle on $X$. Then $\SB(L)=\SB(L|_{\mathbb{E}(L)})$ so long as $L|_{X_{\mathbb{Q}}}$ is semiample.
\end{corollary}
\begin{proof}
	By Noetherian induction we may suppose that this holds on every proper closed subscheme. By \autoref{BaseRed} we may suppose that $X$ is reduced and then we may also assume $\mathbb{E}(L) \neq X$, else the result is trivial. Let $X'$ be the union of components on which $L$ is big and $X''$ the union of those on which it is not.
	
	Let $A$ be an ample line bundle and $s$ a general section of $mL-A$, then $Z=Z(s)$ must contain $\mathbb{E}(L)$. By \autoref{bigsecs} we have that $Z \neq X$, since $s$ does not vanish on any component of $X'$. Since $\mathbb{E}(L|_{Z}) = \mathbb{E}(L)\cap Z= \mathbb{E}(L)$ we must have $\SB(L)=\SB(L|_{Z})=\SB(L|_{\mathbb{E}(L)})$ by the induction hypothesis. 
\end{proof}

\section{Augmented Base Loci}

This section considers the augmented base locus of a nef line bundle and its relation to the exceptional locus. This is done largely under the assumption that they are equal in characteristic $0$, before showing this assumption is satisfied in two key cases.

\begin{lemma}\label{amp}
	Let $X$ be a projective scheme, $L$ a line bundle and $A$ a very ample line bundle. Then for $m>>0$ large and divisible we have that 
	
	$$\BS(L)=\BB(mL-A).$$
\end{lemma}

\begin{proof}
	Certainly we have $n$ such that $\BS(L)=\SB(nL-A)$ and thus also $\BS(L)=\BB(nkL-kA)$ for large divisible $k$. Conversely however $\BB(nkL-A)\subseteq \BB(nkL-kA)$ as $A$ is very ample. Since $\BS(L)\subseteq \BB(nkL-A)$ by definition, taking $m=kn$ suffices.
\end{proof}

\begin{lemma}
	Let $X$ be a projective scheme over an excellent Noetherian base with $L$ a nef line bundle on $X$. If $D$ is an effective Cartier divisor with $L(-D)$ an ample line bundle and $\BS(L|_{kD})=\BS(L|_{D})$ for all $k > 0$ then $\BS(L)=\BS(L|_{D})$.
\end{lemma}

\begin{proof}
	Since $D=L-A$ we must have that $\BS(L) \subseteq D$, and conversely $\BS(L|_{D}) \subseteq \BS(L)$ since we may always pullback sections. It suffices to show then that $\BS(L) \subseteq \BS(L|_{D})$ and we need only check this on points inside $D$.
	
	By taking multiples we may freely assume $L-D=2A$ for $A$ very ample. Consider the short exact sequence
	\[0 \to \ox(k(mL-D-A))\to \ox(kmL-kA) \to \ox[kD](mkL-kA) \to 0.\]
	We have that $H^{1}(X,kmL-kD-kA)=H^{1}(X,(k-1)mL+kA)=0$ for $k >>0$ which we now fix and for all $m >0$.
	
	In particular we may lift sections from $\ox[kD](mkL-kA)$ for any $m>0$. By assumption we have $\BS(L|_{kD})=\BS(L|_{D})$ and so we have that $\BS(L|_{D})=\BB((mkL-kA)_{kD})$ for sufficiently large and divisible $m$. Given this choice of $m$ we may lift sections avoiding $\BB((mkL-kA)_{kD})$ and thus $\BS(L) \subseteq \BS(L|_{D})$.
\end{proof}

\begin{lemma}\label{reduce}
	Let $X$ be a projective scheme over an excellent Noetherian base with $L$ a nef line bundle on $X$ and $A$ an ample line bundle. If $Z=Z(s)$ for some $s$ a section of $mL-A$ and $\BS(L|_{kZ})=\BS(L|_{Z})$ for all $k \geq 1$ then $\BS(L)=\BS(L|_{Z})$.
\end{lemma}

\begin{proof}
	As above we need only prove that $\BS(L)\subseteq \BS(L|_{Z})$.	Let $Y_{1}$ the union of components on which $Z$ is non-zero and $Y_{2}$ the union of those on which it is not. From above we may assume that $Y_{1} \neq \emptyset$ else $Z_{red}=X_{red}$ and the result follows. Let $D=Z|_{Y_{1}}$ and write $L|_{Y_{1}}=L', A|_{Y_{1}}=A'$. As in the proof of previous theorem, after possibly replacing $L,D$ with a multiples, we may find $k$ such that every section of $(mkL'-kA')|_{kD}$ lifts to one of $mkL'-kA'$. 
	
	Similarly for $n>>0$ sufficiently divisible we have $\mathbf{B}((nL-kA)|_{kZ})=\BS(L|_{kZ})=\BS(L|_{Z})$ by assumption. Taking any section $s$ of $(mkL-kA)|_{kZ}$, we may restrict to a section on $kD$ and then lift to $s'$ a section of $k(mL'-A')$. By construction $s|_{Y_{2}},s'$ glue along $Y_{1}\cap Y_{2}\subseteq D$ to give a corresponding section of $k(mL-A)$ and the result follows. We may perform this gluing by \autoref{powers}.
\end{proof}

\begin{lemma}\label{red-eq}
	
	Let $X$ be a projective scheme over an excellent Noetherian base with $L$ a nef line bundle on $X$. Suppose that $\BS(L)=\mathbb{E}(L)$ and that $Z$ is closed subscheme of $X$ with $\mathbb{E}(L) \subseteq Z$. Then $\BS(L|_{Z})=\mathbb{E}(L|_{Z})$.
	
\end{lemma}

\begin{proof}
	
	Choose $m> 0$, and $A$ ample on $X$ with $\BS(L)=\BB(mL-A)$ and $\BS(L|_{Z})=\BB((mL-A)|_{Z})$. Then we have that $\BB((mL-A)|_{Z}) \subseteq \BB(mL-A)\cap Z$ by restriction.
	
	On the other hand, since $\mathbb{E}(L) \subseteq Z$, we have that $\mathbb{E}(L|_{Z}) = \mathbb{E}(L)$. Hence we have that $$\BS(L|_{Z})\subseteq \BB((mL-A)|_{Z})\subseteq \BB(mL-A)\cap Z =\mathbb{E}(L) \cap Z =\mathbb{E}(L|_{Z}).$$
	It is always the case that $\mathbb{E}(L|_{Z}) \subseteq \BS(L|_{Z})$ and hence equality holds.

\end{proof}

\begin{theorem}\label{ext}
	
	Let $X$ be a projective scheme over an excellent Noetherian base $S$ with $L$ a nef line bundle on $X$. Suppose that $\BS(L|_{X_{\mathbb{Q}}})=\mathbb{E}(L|_{X_{\mathbb{Q}}})$. Then in fact $\BS(L)=\mathbb{E}(L)=\BS(L|_{X_{red}})$.
	
\end{theorem}

\begin{proof}
	
	It is immediate that $\mathbb{E}(L)\subseteq \BS(L)$. Since $\mathbb{E}(L)=\mathbb{E}(L|_{X_{red}})$ it suffices to show only that $\BS(L) \subseteq \mathbb{E}(L)$. We may assume therefore that $\mathbb{E}(L) \neq X$ and $L$  is big, or the result follows immediately.
	
	The proof will be by Noetherian induction. So we assume that the result holds on every proper closed subscheme of $X$. The question is local on the base, so we may assume that $S$ is a $\mathbb{Z}_{(p)}$ scheme for some $p > 0$. Note that by \autoref{red-eq} we have that $\mathbb{E}(L|_{X_{red,\mathbb{Q}}})=\BS(L|_{X_{red,\mathbb{Q}}})$\\
	\\
	\textbf{Step 1: Find a non-vanishing section $t$ of $mL-A$.}\\
	
	Take $A$ ample and $m > 0$ with $\SB(mL-A)=\BS(L)$ and $\SB((mL-A)|_{X_{red}})=\BS(L|_{X_{red}})$. Then we have $\mathbb{E}(L|_{X_{red,\mathbb{Q}}})=\SB((mL-A)|_{X_{red,\mathbb{Q}}})$ also. Suppose first that $\SB((mL-A)|_{X_{\mathbb{Q}}}) \neq X_{\mathbb{Q}}$. Then there is some non-zero section $t$ of $mL-A$ which does not vanish everywhere on $X_{red}$.
	
	Otherwise we have $\mathbb{E}(L)=\SB((mL-A)|_{X_{red,\mathbb{Q}}}) = X_{red}$, that is $$H^{0}(X_{red,\mathbb{Q}},k(mL-A)|_{X_{red,\mathbb{Q}}}) =0$$ for all $k$. Since $\mathbb{E}(L|_{X_{red}})=\mathbb{E}(L) \neq X$, $L|_{X_{red}}$ is still big. Now by \autoref{bigsecs} there is a section $s\in H^{0}(X_{red},(mL-A)|_{X_{red}})$ which does not vanish on any component on which $L|_{X_{red}}$ is big. In particular it does not vanish everywhere. Then since $H^{0}(X_{red,\mathbb{Q}},(mL-A)|_{X_{red,\mathbb{Q}}}) =0$ we may use \cite[Proposition 3.5]{witaszek2020keel} to lift $s$ to a section $t$ of $H^{0}(X,p^{e}(mL-A))$ for some $e > 0$ with $t|_{X_{red}}=s^{p^{e}}$. After replacing $L$ and $A$ with their $p^{e^{th}}$ powers, $t$ is precisely the non-vanishing section we seek.\\
	\\
	\textbf{Step 2: Reduce to $Z=Z(t)$.}\\
	
	By construction we have $\mathbb{E}(L)\subseteq Z$, since $\BS(L) \subseteq Z$. By \autoref{red-eq}, then, we have that $\BS(L|_{kZ_{\mathbb{Q}}})=\mathbb{E}(L|_{kZ_{\mathbb{Q}}})$ for $k \geq 1$, so the hypotheses of the theorem are still satisfied by $kZ$. Hence by the induction hypotheses we may assume $\BS(L|_{kZ})=\mathbb{E}(L|_{kZ})=\BS(L|_{Z_{red}})$ for all $k \geq 1$. Therefore we can apply \autoref{reduce} to deduce the result.
\end{proof}

\begin{remark}
	It is not clear in what generality the assumptions of this theorem should hold. Certainly if $S_{\mathbb{Q}}$ is a field they hold by \cite{birkar2017augmented}. Even when $S_{\mathbb{Q}}$ is of finite type over a field however it is not known whether the condition holds. The arguments of \cite{birkar2017augmented} do not hold in this relative setting as they rely heavily on certain cohomology groups being vector spaces over a field. One possible remedy, when $S_{\mathbb{Q}}$ is of finite type over a field, is to find a suitable compactification and reduce to the case that $X_{\mathbb{Q}}$ is projective over a field.
\end{remark}

\begin{lemma}\label{Case:SA}
	Let $X$ be a projective scheme over an excellent base $S$. Suppose that $L$ is a semiample line bundle, inducing $\pi:X \to Y$ with $\pi_{*}\ox=\mathcal{O}_{Y}$. Then we have equalities \[\mathbb{E}(L)=\BS(L)=\text{Exc}(\pi)\]
	where $\text{Exc}(\pi)$ is the union of closed, integral subschemes $Z \subseteq X$ such that $Z \to \pi(Z)$ is not an isomorphism at the generic point.
\end{lemma}

\begin{proof}
	
	The morphism $\pi$ is proper and it's own Stein factorisation. So by Zariski's Main Theorem \stacks{03GW}, $\text{Exc}(\pi)$ is precisely the complement of the locus on which $\pi$ is finite, or equally the locus on which it has finite fibres.
	
	After replacing $L$ with a multiple we have $L=\pi^{*}A$ for some ample $A$ on $Y$.
	
	Take any hyperplane $H$ on $X$, let $\mathcal{I}=\pi_{*}\ox(-H)$ be the ideal sheaf induced on $Y$, so that we have $\pi_{*}(\ox(kL-H))=\ox[Y](kA)\otimes \mathcal{I}$. 
	
	Suppose that $x \in X \setminus \text{Exc}(\pi)$, then we may assume $H$ does not contain $x$ and so the co-support of $I$ does not contain $\pi(x)$. Choose $k>>0$ such that $\ox[Y](kA)\otimes I$ is globally generated. Hence there is a section $s \in H^{0}(Y,\ox[Y](kA)\otimes \mathcal{I})$ not vanishing at $\pi(x)$.
	
	However by adjunction we have natural isomorphisms $$H^{0}(Y,\ox[Y](kA)\otimes \mathcal{I}) \simeq H^{0}(Y,\pi_{*}(\ox(kL-H)))\simeq H^{0}(X,kL-H).$$ The corresponding section $s' \in H^{0}(X,kL-H)$ does not vanish at $x$ by construction.
	
	Hence we have inclusions $\mathbb{E}(L)\subseteq \BS(L)\subseteq \text{Exc}(\pi)$ and it remains to show that $\text{Exc}(\pi) \subseteq \mathbb{E}(L)$. More precisely it is enough to show that if $V$ is any closed, integral subscheme of $X$ such that $L|_{V}$ is big then $V \to \pi(V)$ is generically an isomorphism. 
	
	Suppose then that $L'=L|_{V}$ is big, so we have a section $s$ of $kL'-A$ for $k>>0$ and $A$ ample on $V$. Since $V$ is integral, by assumption, this induces an inclusion $\ox[V](A) \hookrightarrow \ox[V](kL')$. Now $\pi_{V}:V \to \pi(V)$ is generically an isomorphism if and only if it is generically finite, and hence if and only if it's Stein factorisation is so. Therefore we may freely replace $\pi_{V}$ with its Stein factorisation and assume that $\pi_{V}$ is induced by generating sections of $kL'$. Then the inclusion $\ox[V](A) \hookrightarrow \ox[V](kL')$ ensures that $\pi_{V}$ is generically an isomorphism, completing the proof.
\end{proof}

\begin{corollary}\label{main1}
	Let $X$ be a projective scheme over an excellent Noetherian base $S$ with $L$ a nef line bundle on $X$. 
	Suppose that one of the following holds:
	\begin{enumerate}
		\item $S_{\mathbb{Q}}$ has dimension $0$;
		\item $L|_{X_{\mathbb{Q}}}$ is semiample;
	\end{enumerate}
	Then $\BS(L)=\mathbb{E}(L)$.
\end{corollary}

\begin{proof}

	By \autoref{ext}, it is enough to know $\BS(L|_{X_{\mathbb{Q}}})=\mathbb{E}(L|_{X_{\mathbb{Q}}})$. In case $(1)$ this follows from \cite[Theorem 1.3]{birkar2017augmented}, since each connected component of $X_{\mathbb{Q}}$ is projective over a field. In case $(2)$ this is the content of \autoref{Case:SA}.
\end{proof}

\printbibliography

\end{document}